\documentclass[11pt]{amsart}
\usepackage{amsmath,amssymb,amscd}
\usepackage[matrix,arrow]{xy}

\begin{document}

\theoremstyle{plain}
\newtheorem{theo}{Theorem}[section]
\newtheorem{lem}[theo]{Lemma}
\newtheorem{sublem}[theo]{Sublemma}
\newtheorem{prop}[theo]{Proposition}
\newtheorem{conj}[theo]{Conjecture}
\newtheorem{coro}[theo]{Corollary}

\theoremstyle{definition}
\newtheorem{defi}[theo]{Definition}
\newtheorem*{axiom}{Axiom \Roman{zaehler}\stepcounter{zaehler}}

\theoremstyle{remark}
\newtheorem{rem}[theo]{Remark}
\newtheorem{hisrem}[theo]{Historical Remark}
\newtheorem{exam}[theo]{Example}

\newcommand{\lr}{\longrightarrow}
\newcommand{\p}{\pi^{ab}_1}
\newcommand{\Gal}{\mathrm{Gal}}
\newcommand{\chara}{\mathrm{char}}
\newcommand{\Z}{\mathbb{Z}}
\newcommand{\Spec}{\mathrm{Spec}}
\newcommand{\I}{\mathrm{I}}
\newcommand{\C}{\mathrm{C}}
\newcommand{\coker}{\mathrm{coker}}
\newcommand{\im}{\mathrm{im}}
\newcommand{\OO}{\mathcal{O}}
\newcommand{\Nis}{\text{\rm Nis}}
\newcommand{\et}{\text{\rm \'et}}
\newcommand{\sI}{\mathcal I}
\newcommand{\sF}{\mathcal F}
\newcommand{\Sch}{\mathcal S}
\newcommand{\di}{{\rm d}}

\author{Moritz Kerz} 
\address{Moritz Kerz\\
Universit\"at Duisburg-Essen\\
Fachbereich Mathematik, Campus Essen\\
45117 Essen\\
Germany}
\email{moritz.kerz@uni-due.de}

\title{
Ideles in higher dimension}

\date{4/6/11 \hspace{1cm}   MSC: 19F05, 11R37}

\maketitle

\begin{abstract}
We propose a notion of  idele class groups of finitely generated fields using the concept of relative Parshin chains. This new class group allows us to
give an idelic interpretation of the higher class field theory of Kato and Saito.
\end{abstract}

\section{Historical Background and Introduction}

\noindent
Let $F$ be a number field with ring of integers $\OO_F$, and let $X=\Spec(\mathcal{O}_F)$.
The classical henselian finite idele group of  $F$
 is
\begin{equation}\label{clidele}
\I_{\rm cl}(F)= \lim_{\stackrel{\lr}{S\subset |X|}} \left( \prod_{v\in S}  F_v^\times \times \prod_{v\notin S} \mathcal{O}_{F_v}^\times \right) ,
\end{equation}
where the direct limit is over all finite subsets $S\subset |X|$ and $F_v$ is the henselization of $F$ at $v$. Here $|X|$ denotes the set of closed points of the scheme $X$.
The corresponding classical class group is defined to be 
\begin{equation}\label{clidele2}
\C_{\rm cl}(F) = \I_{\rm cl}(F) / F^\times .
\end{equation}
The properties of these topological group, or more precisely their complete analogs,
were first studied by Chevalley and Weil in the 1930s.

Parshin \cite{Parshin} was the first to realize that one has to use
higher Milnor $K$-theory \cite{Milnor} in order to generalize these idelic groups to higher dimensions. He was able to write down an idele class group
for two-dimensional schemes using a concept nowadays called Parshin chains.
Parshin chains are higher dimensional analogs of the places of a one-dimensional global field.
 A few years later a slightly different two-dimensional class group was studied in
detail by Bloch \cite{Bloch} and by Kato and Saito who proved basic properties of two-dimensional 
class field theory; for an overview see \cite{KS1}.
For the last thirty years people tried to generalize such an idelic class group to dimensions greater than two. 
In fact Beilinson had generalized adeles to arbitrary dimensions \cite{Bei}, but the theory of ideles remained mysterious, see~\cite{Raskind} Sec.\ 6.2.

Without defining explicit idelic class groups Kato and Saito developed 
class field theory of higher dimensional schemes using a certain Nisnevich cohomology group of some Milnor $K$-sheaf \cite{KS2}, see also Sections \ref{CompNis} and \ref{CFT} of this note.

In an attempt to strengthen and simplify Kato's and Saito's class field theory Schmidt, Spie\ss , Wiesend and the author developed a version of the class group in higher dimensions which does not use Milnor $K$-theory, see \cite{SchmidtSpiess}, \cite{Wiesend}, \cite{KSch} and \cite{Kerz}. This class group is very explicit, but one of its  disadvantages is that one does not know how to treat wild ramification over a finite field using it.

\medskip

In this note we present the missing idelic interpretation of the higher class field theory of Kato and Saito. The precise result is stated in Theorem~\ref{Niscompteo2}.
Let us explain the problem: Given an integral normal scheme $X$ of dimension $d$ which is proper over $\Spec(\Z)$, Kato and Saito defined a relative class
group 
\[
\C_{\Nis}(X,\mathcal I)=H^d(X, \mathcal{K}^M_{d}(\mathcal{O}_{X}, \mathcal I) ) .
\]
The argument $\mathcal I$ in the Nisnevich relative Milnor $K$-sheaf $\mathcal{K}^M_{d}(\mathcal{O}_{X}, \mathcal I)$ is a non-vanishing coherent ideal sheaf of $X$, which is the analog of a modulus of ramification in classical class field theory. For simplicity we assume that the function field $F$ 
of $X$ has no real embedding $F\hookrightarrow \mathbb R$. Higher class field theory is the study of the reciprocity map
\[
\rho: \C_\Nis (X,\mathcal I) \lr \Gal (F^{ab}_{\mathcal I}/F)
\]
where $F^{ab}_{\mathcal I}$ is the maximal abelian extension of $F$ whose ramification is bounded by $\mathcal I$ in an appropriate sense.
For a summary of the main results of Kato and Saito in terms of our new idelic class group see Section~\ref{CFT}.

In their work~\cite{KS2} Kato and Saito give an idelic description of the dual of the class group $\mathrm{Hom}(\C_\Nis (X,\mathcal I) , A ) $ for an 
arbitrary abelian group $A$ in terms of an {\em infinite product} over all Parshin chains of $X$. If one dually tries to find an actual presentation 
of $\C_\Nis (X,\mathcal I)$ in terms of the corresponding {\em direct sum} over all Parshin chains one faces the problem that the reciprocities which correspond
to two-dimensional local rings on $\mathrm{supp}(\mathcal O_X /\mathcal I)$ do not lie in the direct sum. If $\dim(X)\le 2$ there are no such two-dimensional local
rings, which explains why it is much easier to define idelic class groups in this case.

Our new idea to solve the problem is roughly: Do not factor out the reciprocities that cause the problem. Then what remains is to show that you actually get the same class
group as Kato and Saito, which is done in Section~\ref{CompNis}.

\medskip

Section 1 is motivational. In Sections~2 -- 6 we define our idele class group and prove basic properties. Section~7 shows how Wiesend's class group is related to our ideles. In Section~8 we prove the central theorems of this note giving an isomorphism between the Kato-Saito class group and our idele class group.
Section~9 states the main results of higher class field theory in terms of our idele class group.

\medskip

In this note we restrict ourselves to introducing the new idele class groups and giving their first properties. In \cite{KeSai} we will show how our idelic approach leads to new and direct proofs for practically all known types of higher dimensional class field theory. We also obtain new results, e.g.,\ on the class field theory of henselian local rings.

\medskip

I would like to thank Uwe Jannsen, Shuji Saito and Alexander Schmidt for numerous discussions on higher class field theory. Reading Kanetomo Sato's article~\cite{Sato}
was very enlightening  for the preparation of the present note. Alexander Schmidt made helpful suggestions for improving the text.

\section{Some motivation: inverse limits versus  restricted products}

\noindent
As we recalled in Section~1 the class group of a number field is usually defined using restricted products. It is fundamental
to our higher dimensional generalizations to define the class group in terms of inverse limits instead of restricted products. In order to motivate
this approach we discuss the one-dimensional case in this section. For simplicity we restrict to non-archimedean ideles.

Let $F$ be a number field and let $X$ be 
$\Spec(\mathcal{O}_F)$. In (\ref{clidele}) and (\ref{clidele2}) we recalled  the classical idele group $\I_{\rm cl}(F)$ and the classical idele class group $\C_{\rm cl} (F)$.
For us a different relative description of these groups will be pivotal.
\begin{defi}
For a nonempty open subscheme  $U\subset X$ the relative idele group is defined as
\[
\I (U\subset X)  = \bigoplus_{x\in |U|} \Z \oplus \bigoplus_{v\in X\backslash U} F_v^\times 
\]
with the direct sum topology.
The relative idele class group is defined as
\[
\C (U\subset X) = \coker [F^\times \to \I ( U \subset  X)]
\]
with the quotient topology.
\end{defi}

To formalize our approach consider the category of pairs $U\subset X$, where $U$ is a nonempty open subscheme of a normal connected one-dimensional proper
scheme $X$ over $\Spec(\Z)$. A morphism $f:(V\subset Y)\to (U\subset X)$ in this category is a dominant morphisms of schemes $f:Y\to X$ such that 
$f(V)\subset U$.
It is easy to check that $\I$ and $\C$ form covariant functors from this category to the category of topological abelian groups.

\begin{defi}
For a global one-dimensional field $F$ and $X=\Spec(\mathcal{O}_F)$ we define the idele group as the topological group
\[
\I (F) = \lim_{\longleftarrow \atop U} \I (U\subset X)
\]
and the idele class group as the topological group
\[
\C (F) = \lim_{\longleftarrow \atop U} \C (U\subset X) ,
\]
where the inverse limit is over all nonempty open subschemes $U\subset X$.
\end{defi}

The central observation of this section is that the idele and class groups defined in Section~1 coincide with the groups defined here.

\begin{prop}
The natural maps
\begin{equation}\label{compaiso}
\I_{\rm cl} (F) \to \I(F) \;\;\;\text{ and }\;\;\; \C_{\rm cl} ( F)\to \C (F)
\end{equation}
are isomorphisms of topological groups.
\end{prop}

\begin{proof}
Indeed, for arbitrary finite $S\subset |X|$ and a nonempty open subscheme $U\subset X$ there are continuous homomorphism
of topological groups
\[
 \prod_{v\in S}  F_v^\times \times \prod_{v\notin S} \mathcal{O}_{F_v}^\times    \lr    \I (U\subset X)
\]
and 
\[
\left( \prod_{v\in S}  F_v^\times \times \prod_{v\notin S} \mathcal{O}_{F_v}^\times  \right)/ \mathcal{O}_{F,S}^\times   \lr    \C (U\subset X).
\]
Taking the direct limit on the left side and the inverse limit on the right side we get the continuous maps in (\ref{compaiso}).
It is elementary to check that the first map in (\ref{compaiso}) is an isomorphism. But one checks that
\[
0\to F^\times \to \I (F) \to \C (F) \to 0
\]
is topologically exact, so the second map in (\ref{compaiso}) is also an isomorphism.
\end{proof}

Our aim in the following sections is to transfer this procedure to the definition of an idele class group of a finitely generated field $F$. Namely,
we will define a relative idele class group of a proper model $X$ of $F$ relative to some open subscheme $U\subset X$  using only direct sums.
Then we take the inverse limit over all dense open subschemes $U\subset X$ in order to define the idele class group of $F$.


\section{Parshin chains}

\noindent
In this note we will work with the category $\Sch$ of integral excellent schemes which are endowed with a dimension function $\di$.
By a dimension function on a scheme $X$ we mean a set theoretic funtion $\di :X \to \Z$ which statisfies the following properties
\begin{itemize}
\item for all $x\in X$ we have $\di(x)\ge 0$
\item for $x,y\in X$ with $y\in \overline{ \{x\} }$ of codimension one we have $\di(x)=\di(y)+1$.
\end{itemize}
The morphisms in the category $\Sch$ from object $(X',d')$ to object $(X,d)$ will be quasi-finite morphisms $f:X' \to X$ with $\di'(x)=\di(f(x))$
for all $x\in X'$.

In the following let $(X,\di)$ be an object of $\Sch$.
Fix an effective Weil divisor $D$ of $X$. Set $U=X-D$ and $F=k(X)$. For more details about Parshin chains we refer to~\cite[Section 1.6]{KS2}.
Let $d_{\rm m}$ be the minimum of the integers $\di(x)$ for $x\in X$.

\begin{defi}\mbox{} 
\begin{itemize}
\item A chain  on $X$ is a sequence of points $P=(p_0,\ldots , p_s)$ of $X$ such that 
\[
\overline{\{p_0 \}} \subset \overline{\{p_1 \}} \subset \cdots \subset \overline{\{p_s \}}  . 
\]
\item A Parshin chain on $X$ is a chain $P=(p_0,\ldots , p_s)$ on $X$ such that $\di(p_i)=i+d_{\rm m}$ for $0\le i\le s$.
\item A Parshin chain on the pair $(U\subset X)$ is a Parshin chain $P=(p_0,\ldots , p_s)$ on $X$ such that $p_i\in D$ for $0\le i<s$ and such that
$p_s\in U$.
\item The dimension $\di(P)$ of a chain $P=(p_0,\ldots , p_s)$ is defined to be $\di(p_s)$.
\end{itemize}
\end{defi}

Using repeated henselizations one defines the finite product of henselian local rings $\OO^h_{X,P}$ of a chain $P=(p_0,\ldots , p_s)$ on $X$ as follows:
If $s=0$ set $\OO^h_{X,P}= \OO^h_{X,p_0}$. If $s>0$ 
assume that $\OO^h_{X,P'}$ has already been defined for chains of the form $P'=(p_0,\ldots , p_{s-1})$.
Given a chain $P=(p_0,\ldots , p_s)$ let $P'$ be the chain $(p_0,\ldots , p_{s-1} )$. Denote the ring $\OO^h_{X,P'}$ by $R$ and
 denote by $T$ the finite set of primes ideals of $R$ lying over $p_s$. Then we set
$$\OO^h_{X,P}  = \prod _{\mathfrak p \in T} R^h_{\mathfrak p}.$$
Finally, for a chain $P=(p_0,\ldots , p_s)$ on $X$ we let $k(P)$ be the finite product of the residue fields of $\OO^h_{X,P}$. If $s\ge 1$ each of these residue fields has a natural discrete valuation such that the product of their rings of integers is equal to the normalization of
$ \OO^h_{X,P'} / p_s  $, where $P'= (p_0,\ldots , p_{s-1}) $.

\begin{defi}\label{liftchain}
A lifted chain on $X$ is a sequence $P=(p_0,\ldots , p_s)$ such that $p_0\in X$ and $p_{i+1}\in \Spec (\mathcal O^h_{(p_0,\ldots ,p_i)}) $,
where the latter henselian spectrum is defined as a successive henselization as above.
\end{defi}


\section{Idele groups}\label{IdeleGroups}

\noindent 
In this section we define higher dimensional relative idele groups. Let $(X,\di)$ be an object of $\Sch$. Set $d=\di(\eta)$, where $\eta\in X$
is the generic point, and let $d_{\rm m}$ be the minimum of the numbers $\di(x)$ with $x\in X$. Furthermore, $D$ will be an effective Weil divisor on $X$ and $U=X-D$.

We recall the definition of Milnor $K$-theory.
\begin{defi}
For a commutative unital ring $R$ let $T(R^\times)$ be the tensor algebra over the units of $R$. Let $I\subset T(R^\times )$ be the
two-sided ideal generated by elements $a\otimes (1-a)$ with $a,1-a\in R^\times$. Then the Milnor $K$-ring of $R$ is the graded ring
\[
K^M_*(R)=T(R^\times ) /I .
\]
The image of $a_1\otimes \cdots \otimes a_n$ in $K^M_n(R)$ is denoted by $\{a_1,\ldots ,a_n \}$.
\end{defi}
\begin{defi}\label{MilTop}
If $R$ is a discrete valuation ring with quotient field $F$ and maximal ideal $\mathfrak p \subset R$ we let
$K^M_n(F,m)\subset K^M_n(F)$ be the subgroup generated by $\{1+ \mathfrak p^m , R^\times ,\ldots ,R^\times \}$
for an integer $m\ge 0$. The topology of $K^M_n(F)$ generated by these subgroups will be called the canonical topology.
\end{defi}

Let $\mathcal{P}$ be the set of Parshin chains on the pair $(U\subset X)$. For a Parshin chain $P=(p_0,\ldots ,p_{d-d_{\rm m}})\in\mathcal{P}$ of dimension $d$ let
$D(P)$ be multiplicity of the prime divisor $\overline{\{ p_{d-1} \}}$ in $D$.

\begin{defi}
The idele group of the pair $(U\subset X)$ is defined as
\[
\I (U\subset X)= \bigoplus_{P\in \mathcal{P}} K^M_{\di(P)}(k(P)) .
\] 
We endow this group with the topology generated by the open subgroups
\[
\bigoplus_{P\in \mathcal{P}\atop \di(P)=d}  K^M_{d}(k(P),D(P)) \subset \I (U\subset X)
\]
where $D$ runs through all effective Weil divisors with support $X-U$.
The idele group of $X$ relative to the fixed effective divisor $D$ with complement $U$ is defined as
\[
\I (X,D) = \coker[\!\!\!\! \bigoplus_{P\in \mathcal{P}\atop \di(P)=d}  K^M_{d}(k(P),D(P))  \to \I (U\subset X) ] .
\]
\end{defi}

The topology on $\I(U\subset X)$ is ad hoc but we will see in Section~\ref{SecClass} that it induces a useful topology on the idele
class group. It is a nice property of higher global class field theory that you do not have to be too careful about topologies, because at the end
you usually get a `good' topology on the idele class group.
This is in sharp contrast with the extremely difficult study of topologies on the class group of a higher local field, i.e. on higher Milnor $K$-groups.


\section{Functoriality and reciprocity} 

\noindent 
Let $(X,\di)$ and $(X',\di')$ be objects of $\Sch$ and let $U\subset X$ and $U'\subset X'$ be nonempty
open subschemes.
If $f:X'\to X$ is a finite  morphism compatible with $\di$ and $\di'$ and if it induces a 
morphism of open subschemes $U'\to U$ 
we define a continuous homomorphism $$f_*:\I(U'\subset X')\to \I(U\subset X)$$ by the following procedure:
Let $\mathcal{P}$ resp.\ $\mathcal{P}'$ be the set of Parshin chains on the pair $(U\subset X)$ resp.\ $(U'\subset X')$.
If $P=(p_0,\ldots ,p_s)\in\mathcal{P}$  and $P'=(p'_0,\ldots ,p'_{s'})\in\mathcal{P}'$  satisfy $s\le s'$ and
$f(p'_i)=p_i$ for $0\le i\le s$ then we let 
\[
f_*^{P'\to P} : K^M_{\di' (P')}(k(P')) \lr K^M_{\di(P)}(k(P))
\]
be the composition of the residue map 
\[
\partial : K^M_{\di'(P')}(k(P')) \lr  K^M_{\di'(P'')}(k(P''))
\]
with $P''=(p'_0,\ldots ,p'_s)$ and the norm map of Milnor $K$-groups
\[
N:K^M_{\di'(P'')}(k(P'')) \lr K^M_{\di(P)}(k(P)) .
\]
Observe that $\di'(P'') =\di(P)$.
For the definition of these maps due to Milnor, Bass-Tate and Kato see for example~\cite[Chapters 7 and 8]{GS}. Otherwise we let $f_*^{P'\to P}$ be the zero map. Then 
$f_*$ is defined as the sum of the maps $f_*^{P'\to P}$ for all $P\in\mathcal{P}$ and $P'\in\mathcal{P}'$.
It follows from~\cite[Lemma 4.3]{KS2} that the map $f_*$ is continuous. It is also easy to see that this construction of a pushforward 
map is functorial. 

\begin{lem}\label{shrink}
If $f:X'\to X$ is an isomorphism and if $U$ is regular the pushforward map
$$f_*:\I(U'\subset X')\lr \I(U\subset X)$$
is surjective and a topological quotient.
\end{lem}

\begin{proof}
In fact for a Parshin chain $P=(p_0,\ldots ,p_s)$ on the pair $(U\subset X)$ we choose successively for $i>s$  points
$p_i$ of $\Spec( \OO_{X,p_{i-1}} )$ such that $\di(p_i)=i+d_{\rm m}$ and such that $\overline{\{p_i\}}\subset \Spec( \OO_{X,p_{i-1}} )$ is regular until we get a Parshin chain $P'=(p_0,\ldots ,p_{s'})$ on the pair $(U'\subset X'=X)$
for some $s'\ge s$. Then the map 
\[
f_*^{P'\to P} : K^M_{\di(P')}(k(P')) \lr K^M_{\di(P)}(k(P))
\]
is surjective.
\end{proof}

We can now define the idele group $\I (F;X)$ of the function field field $F=k(\eta)$ of the scheme $(X,\di)\in \Sch$.
\begin{defi}
We set
\[
\I (F;X) = \lim_{\stackrel{\longleftarrow}{U\subset X}} \I (U\subset X)
\]
where the inverse limit is over all nonempty open subschemes $U\subset X$.
We endow $\I (F;X)$ with the inverse limit topology.
\end{defi}

Next we will consider arithmetic reciprocity maps.
Let $X$ be an integral scheme proper over $\Z$ and define the dimension function to be $\di(x)=\dim( \overline{\{x\} } )$.
Assume for simplicity that $F=k(X)$ has no embedding into the reals,
otherwise we had to take archimedean Parshin chains into account.
It follows from the usual properties of higher local fields \cite[Section 3.7.2]{KS2} that there exists a continuous functorial prereciprocity homomorphism
\[
r : \I (F;X) \lr  \Gal (\bar{F} / F)^{ab}
\]
which has dense image according to Lang's  generalized Cebotarev density theorem, see \cite{Serre}.

\section{Class groups}\label{SecClass}

\noindent
Let $(X,\di)$ be an object of $\Sch$ and let $U\subset X$ be a nonempty open subscheme.
Let furthermore $D$ be an effective Weil divisor with support $X-U$.
By $\mathcal{P}$ we denoted the set of Parshin chains on the pair $(U\subset X)$.

\begin{defi}
A $Q$-chain on $(U\subset X)$ is defined as a chain $P=(p_0, \ldots ,p_{s-2} , p_s)$ on $X$ for $1\le s\le d$ such that $\di(p_i)=i+d_{\rm m}$ for $i\in\{ 0,1,\ldots ,s-2,s \}$
and such that
\begin{itemize}
\item
$p_i\in D$ for $0\le i\le s-2$,
\item  $p_s\in U$.
\end{itemize}
\end{defi}

We denote the set of $Q$-chains on $(U\subset X)$ by $\mathcal{Q}$.
Then there exists a natural homomorphism 
\[ 
Q:\bigoplus_{P\in\mathcal{Q}} K^M_{\di(P)}(k(P)) \lr \I(U \subset X)
\]
defined as follows. If $P'=(p_0, \ldots , p_{s-2},p_s)$ is a $Q$-chain and $P=(p_0, \ldots , p_{s-2},p_{s-1} ,p_s)$ is a Parshin chain on $X$
we let $Q^{P'\to P}$ be defined as
\begin{itemize}
\item the map $K^M_{\di(P')}(k(P')) \to K^M_{\di(P)}(k(P))$ induced on Milnor $K$-groups by the ring homomorphism $k(P')\to k(P)$ if $p_{s-1}\in D$,
\item the residue symbol $K^M_{\di(P')}(k(P')) \to K^M_{\di(P'')}(k(P''))$ where $P''=(p_0, \ldots , p_{s-1})$ if $p_{s-1}\in U$.
\end{itemize}
We let the map $Q$ be the sum of all these $Q^{P'\to P}$.

\begin{defi}\label{ideleclassdefi}
The idele class group $\C(U\subset X)$ is defined as the cokernel of $Q$ endowed with the quotient topology. The idele class group of $X$
relative to the effective divisor $D$ is defined as the analogous quotient
\[
\C (X,D)=\coker [ \bigoplus_{P\in\mathcal{Q}} K^M_{\di(P)}(k(P))
\to \I (X,D) ].
\]
\end{defi}

If $f:X'\to X$ is a finite dominant morphism between objects $(X',\di')$ and $(X,\di)$ of $\Sch$ and if it induces a 
morphism of open subschemes $U'\to U$ as in the previous section we get an induced continuous homomorphism $$f_*:\C(U'\subset X') \lr \C(U\subset X)  .$$

\begin{defi}
Let the class group $\C (F;X)$ of $F$ with respect to the integral scheme $X$ proper over $\Z$
 with $k(X)=F$ be defined as 
\[
\C(F;X)= \lim_{\stackrel{\longleftarrow}{U}} C(U\subset X) .
\]
\end{defi}

The continuous homomorphism $\I(F;X) \to \C(F;X)$ has dense image, but it is not clear whether it is surjective.

The next lemma is an immediate consequence of Lemma~\ref{shrink}.
\begin{lem}\label{shrinkclass}
If   $U$ is regular and $U'\subset U$  then the map 
$$\C(U'\subset X)\lr \C(U\subset X)$$
is surjective and a topological quotient.
\end{lem}

As we announced in Section~\ref{IdeleGroups} the  induced topology on $\C(F;X)$ has nice properties. A first result
 which justifies this is the
following. Its proof is similar to the proof of~\cite[Thm.~2.5]{KS2}. We note that the analogous statement is incorrect for the idele group with its rather
naive topology.

\begin{theo}\label{ContImm}
Let $X'\hookrightarrow X$ be a closed integral subscheme such that $U'=U\cap X'\ne\emptyset$ and assume $U$ is regular. Then the canonical map
\[
\C (U'\subset X') \lr \C (U\subset X)
\]
is continuous.
\end{theo}

\begin{proof}
Let $\eta$ resp.\ $\eta'$ be the generic point of $X$ resp.\ $X'$. If  $\di(\eta)-\di(\eta')>1$ choose a point $x \in U - U'$ with $x\ne \eta$ such that the generic point of $X'$ lies on the regular locus of $\overline{\{x\}}$. By descending induction on 
$\di(\eta)-\di(\eta')$ we can replace $X$ by $\overline{\{x\}}$. In order to guarantee that $U\cap \overline{\{x\}}$ is regular we have to shrink $U$, which
is feasible because of Lemma~\ref{shrinkclass}. So assume without loss of generality $\di(\eta)-\di(\eta')=1$. Then it suffices to show:
\begin{lem}
Let $A$ be a two-dimensional excellent normal henselian local rings with maximal ideal $\mathfrak m$ and quotient field $F$. 
 Let $\mathfrak q ,\mathfrak{p}_1 ,\ldots , \mathfrak p_r$
be distinct prime ideals of codimension one. Set $S=\{\mathfrak{p}_1 ,\ldots , \mathfrak p_r \}$. Given integers $m_1,\ldots m_r\ge 0$ there exists $m\ge 0$ such that 
\[
\phi: K^M_{n-1}(k(\mathfrak q), m) \lr  \C (A, \mathfrak p_1^{m_1} \cdots \mathfrak p_r^{m_r} )
\]
is the zero map, where 
\[
\C (A, \mathfrak p_1^{m_1} \cdots \mathfrak p_r^{m_r} ) = \coker [ K^M_n(F) \lr \bigoplus_{\mathfrak p_i\in S}K^M_n(F_{\mathfrak p_i})/  K^M_n(F_{\mathfrak p_i} , m_i) \oplus
\bigoplus_{\mathfrak p\notin S} K^M_{n-1}(k(\mathfrak p ))]  .
\]
Also we can choose $m$ large enough so that $\phi$ remains the zero map if we replace $A$ by a finite \'etale extension.
\end{lem}
\begin{proof}
Choose $\pi\in A$ with $v_{\mathfrak q}(\pi)=1$ and $v_{\mathfrak p_i}(\pi)=0$ for $1\le i\le r$. Let $\mathfrak p_{r+1} ,\ldots , \mathfrak p_t$ be the prime ideals of codimension $1$ containing $\pi$ 
and different from $\mathfrak q$. Let $\pi_{\mathfrak q}$ be a prime element of the discretely valued field $k(\mathfrak q)$.
Choose $\pi_i\in A-\mathfrak q $ with $v_{\mathfrak p_i}(\pi_i)=1$ for $1\le i\le t$ and let $m_i=1$ for $r< i\le t$.
Now choose $m\ge 0$ such that 
\begin{equation}\label{decoeq}
\pi_{\mathfrak q}^m \OO_{k(\mathfrak q)} \subset \pi_1^{m_1} \cdots \pi_t^{m_t} A/\mathfrak q  .
\end{equation}
Given a symbol $\{ \bar{a}_1, \bar a_2,\ldots , \bar a_{n-1} \}\in K^M_{n-1}(k(\mathfrak q), m)$ with $\bar a_1\in 1+\pi_{\mathfrak q}^m \mathcal \OO_{k(\mathfrak q)}$
lift $\bar a_1$ according to (\ref{decoeq}) and lift the other $\bar a_i$ to $A$ arbitrarily. Then one sees that the image of
\[
\{\pi , a_1 ,\ldots ,a_{n-1} \} \in K^M_n(F)\;\;\; \text{ in }\;\;\;  \bigoplus_{\mathfrak p_i \in S}K^M_n(F_{\mathfrak p_i})/  K^M_n(F_{\mathfrak p_i} , m_i) \oplus
\bigoplus_{\mathfrak p\notin S} K^M_{n-1}(k(\mathfrak p ))
\]
is equal to $\{ \bar{a}_1, \bar a_2,\ldots , \bar a_{n-1} \}  \in  K^M_{n-1}(k(\mathfrak q ))$.
\end{proof}

\renewcommand{\qedsymbol}{}

\end{proof}
The next corollary follows immediately from Theorem~\ref{ContImm}.

\begin{coro}\label{corocontimm}
For a Parshin chain $P$ on the pair $(U\subset X)$ with regular $U$ the map
\begin{equation}\label{eqcontmil}
K^M_{\di(P)} (k(P)) \lr  \C (U\subset X)
\end{equation}
is continuous if we endow the Milnor $K$-group of $k(P)$ with the canonical topology of Definition~\ref{MilTop}.
\end{coro}

It seems likely that the stronger fact is true that the map (\ref{eqcontmil}) is continuous with respect to Kato's pseudo topology on Milnor
$K$-groups, see~\cite{Kato}. 

For the next corollary let $X$ be an integral scheme which is proper over $\Z$ and let $U\subset X$ be a dense open subscheme. Let the dimension function on $X$ be given by
$\di(x)=\dim( \overline{ \{ x \}} )$ for $x\in X$.

\begin{coro}\label{denseimage}
The image of the natural map
\[
\bigoplus_{x\in |U|} \Z \lr \C(U\subset X)
\]
is dense if $U$ is regular.
\end{coro}
Here $|U|$ denotes the set of all closed points of $U$.
\begin{proof} Fix $0\le l\le \dim(X)$.
Let $\mathcal{P}_{\le l}$ be the set of Parshin chain $P$ on $(U\subset X)$
of dimension at most $l$ and let $\mathcal{Q}_{\le l}$ be the set of $Q$-chains on $(U\subset X)$ of dimension at most $ l$.
We show by descending induction on $0\le l\le \dim(X)$ that the image $I_l$ of the map
\[
\bigoplus_{P\in\mathcal{P}_{\le l}} K^M_{\di(P)}(k(P)) \lr \C(U\subset X)
\]
is dense. It suffices to show that $I_{l-1}$ is dense in $I_l$. By Corollary~\ref{corocontimm} this follows from the
fact that the image of  the map
\[
\bigoplus_{P\in\mathcal{Q}_{\le l}-\mathcal{Q}_{\le l-1}} K^M_{l}(k(P)) \lr \bigoplus_{P\in\mathcal{P}_{\le l}-\mathcal{P}_{\le l-1}} K^M_{l}(k(P)) 
\]
has dense image where each summand on the right hand side has the canonical topology, see Definition~\ref{MilTop}.
The latter density follows from the following approximation lemma.
\begin{lem}\label{approxlem}
Given a  finite family of inequivalent discrete valuations $v_1,\ldots , v_s$ on $F$ the map
\[
K^M_n(F) \lr \bigoplus_{v_i} K^M_n(F_{v_i})
\]
has dense image, where $F_{v_i}$ is the henselization of $F$ at $v_i$.
\end{lem}
\end{proof}

\section{Wiesend's class group}

\noindent
In this section let $X$ be an integral scheme proper over $\Z$ with $\di(x)=\dim( \overline{ \{ x \}} )$ for $x\in X$.
For an integer $l>0$ denote by $\mathcal{P}_{\le l}$ the subset of the set of Parshin chains $\mathcal{P}$ on the pair $(U\subset X)$ consisting of all chains $P$  with 
$\di(P)\le l$ and let $\mathcal{Q}_{\le l}$ be the set of all $Q$-chains $P$ on $(U\subset X)$ with $\di(P)\le l$.

Replacing $\mathcal{P}$ by $\mathcal{P}_{\le 1}$ and $\mathcal{Q}$ by $\mathcal{Q}_{\le 1}$  in the definition of the idele class group we get Wiesend's class group without the archimedean places.
\begin{defi}
Wiesend's idele group of $U$ is defined as the topological group
\[
\I_W (U) = \bigoplus_{P\in \mathcal{P}_{\le 1}} K^M_{\di(P)} (k(P)) .
\]
Here for $\di(P)=1$ we put the canonical topology on $K^M_1(k(P))=k(P)^\times $.
Wiesend's idele class group is defined as \[\C_W (U) = \coker[\bigoplus_{P\in \mathcal{Q}_{\le 1}}  K^M_{\di(P)}(k(P))\to \I_W(U) ] \]
with the quotient topology.
\end{defi}
Note that Wiesend's idele group and idele class group  of $U$ do not depend on the compactification $U\subset X$.
For more details on these topological groups see \cite{Wiesend}, \cite{KSch} and \cite{Kerz}.
The reader should be warned that the map $\I_{W}(U) \to \I(U\subset X)$ is not continuous for $\dim(X)\ge 2$, but Theorem~\ref{ContImm} shows that the induced map on idele
class groups is continuous.

\begin{prop}
If $U$ is regular the natural homomorphism $\C_W(U)\to \C(U\subset X) $ is continuous.
\end{prop}

One of the most difficult open problems in Wiesend's class field theory is to construct a theory of moduli and to prove existence and finiteness results
for varieties over finite fields. For this it would useful to give 
an explicit description of the kernel of the map
$
\C_W(U) \to \C(X,D)
$
to the idele class group with modulus $D$, at least if $X$ is regular and $D$ is a simple normal crossing divisor.


\section{Comparison Theorem}\label{CompNis}

\noindent
In this section we want to  compare the class group of Kato and Saito, defined using sheaf cohomology, to the idele class group as defined 
in Section~\ref{SecClass}.

Let $K$ be a number field or a finite field. By $\OO_K$ we denote the ring of integers of the number field $K$ or the finite field $K$ itself.
Let $X/\OO_K$ be an integral scheme proper and flat over $\OO_K$. We define a dimension function on $X$ by $\di(x)=\dim( \overline{ \{ x \}} )$ for $x\in X$.

We start with a local comparison theorem. For this we have to develop a local version of the idele class group, which is in complete analogy with
the global theory. Given a lifted chain $P$, see Definition~\ref{liftchain}, on $X$ let $X_P$ be the scheme $\Spec (\mathcal{O}^h_{X,P}) $. 
Let $\mathfrak m_P$ be the maximal ideal of $\OO^h_{X,P}$ and let $x_p$ be the closed point of $X_P$. We endow $X_P$ with the dimension function induced by the dimension function on $X$ via pullback.
Consider an open subset $U_P$ of $X_P$ which is the complement of an effective Weil divisor $D_P$.
Let $X'_P$, $U'_P$ and $D'_P$ be the schemes $X_P-\{x_P\}$, $U_P-\{ x_P \} $ and $ D_P- \{ x_P \}$ .

Recall from Section~\ref{IdeleGroups} that the idele group of the pair $(U'_P\subset X'_P)$ is defined as
\[
\I (U'_P\subset X'_P)= \bigoplus_{P'\in \mathcal{P}} K^M_{\di(P')}(k(P')) .
\] 
Here $\mathcal P$ is the set of Parshin chains on the pair $(U'_P\subset X'_P)$.
The idele class groups $\C (U'_P\subset X'_P)$ and
 $\C (X'_P, D'_P)$ are defined
as  quotients of $\I (U'_P\subset X'_P)$, see Definition~\ref{ideleclassdefi}. By abuse of notation we will write
$\C(U_P \subset X_P)$ for $\C(U'_P \subset X'_P)$ and $\C( X_P,D_P)$ for $\C( X'_P, D'_P)$  in order to simplify the notation. The reader should be warned that with this convention
$\C(U_P \subset X_P)$ and  $\C( X_P,D_P)$  {\em do not come} from the construction of Sections~\ref{IdeleGroups} and \ref{SecClass} for the pair $(U_P\subset X_P)$.

For the local version of the Nisnevich cohomological class group of Kato and Saito we need a relative Milnor $K$-sheaf.
First of all let $\mathcal K^M_*$ be the Nisnevich sheaf associated to the Milnor presheaf defined in Section~\ref{IdeleGroups}.
For an effective Cartier divisor with corresponding  ideal sheaf $\mathcal I_P$  on $X_P$ whose associated Weil divisor is $D_P$ set
\[
\mathcal{K}^M_*(\mathcal{O}_{X_P}, \mathcal I_P) =  \ker[\mathcal{K}^M_*(\mathcal{O}_{X_P}) \to \mathcal{K}^M_*(\mathcal{O}_{X_P}/\mathcal I_P) ]
\]
thought of as a Nisnevich sheaf on $X_P$.

Let $d$ be the Krull dimension of $X_P$, $d_P$ be the dimension of $P$ and $d_X$ be the dimension of $X$.
\begin{defi}
Given an effective Cartier divisor with ideal sheaf $\mathcal I_P$  set
\[
\C_{\Nis}(X_P,\mathcal I_P)=H^{d}_{x_P}(X_P, \mathcal{K}^M_{d_X}(\mathcal{O}_{X_P}, \mathcal I_P) ) .
\]
\end{defi}

This local Kato-Saito class group was also studied by Sato in~\cite{Sato}. 
Note that it gives a useful definition for higher class field theory only if $P$ is a Parshin chain, i.e.~if the residue field 
$k(P)$ is a higher local field. Otherwise there does not exist a reciprocity map to an abelian fundamental group.
 
For an arbitrary lifted chain $P$ and for $x\in (X_P)_1$ there exist natural homomorphisms 
\begin{equation}\label{chainchange}
\C_{\Nis}(X_{(P,x)} ,\mathcal I_{(P,x)}) \to \C_{\Nis}(X_P,\mathcal I_P) \;\;\; \text{ and  }\;\;\; \C (X_{(P,x)} ,D_{(P,x)}) \to 
\C(X_{P} , D_P ). 
\end{equation}
If $X_P$ is (formally) smooth over $\OO_K$ and $D_P=0$ there is a canonical purity isomorphism
\begin{equation}\label{purity}
\C_\Nis (X_P,\mathcal \OO_{X_P}) = \C(X_P,0) = K^M_{d_P}(k(P))
\end{equation}
stemming from the Gersten conjecture as shown in~\cite{KatoChow} Theorem~1 and Theorem~2.
 
Now we can state our local comparison theorem:

\begin{theo}\label{Niscomptheo1}
For lifted chains $P$ on $X$ such that $U_P=X_P- D_P$ is (formally) smooth over $\mathcal{O}_K$ and for  $\mathcal I_P$ the ideal sheaf of an effective Cartier divisor with associated Weil divisor $D_P$ on $X_P$ there exist unique isomorphisms
\[
\C_{\Nis}(X_P,\mathcal I_P) \cong \C (X_P, D_P)
\]
which are compatible with the change of chain morphisms~\eqref{chainchange} and which coincide with the purity isomorphism~\eqref{purity} if $D_P=0$.
\end{theo}

\begin{proof}
Let $\mathcal F$ be the Nisnevich sheaf $ \mathcal{K}^M_{d+d_P+1}(\mathcal{O}_{X_P}, \mathcal I_P) $.
We construct the isomorphism by induction on $d=\di(X_P)$. Uniqueness will be obvious from the construction.
If $d=0$ its definition is clear from~\eqref{purity}. 

If $d=1$ let $\eta_P$ be the generic point of $X_P$ and consider the diagram with exact rows
\[
\xymatrix{
K^M_{d_X}(k(\eta), M) \ar[r] \ar[d]^{\wr} &   K^M_{d_X}(k(\eta ) )  \ar[r]  \ar[d]^{\wr} &  \C (X_P , D_P)  \ar@{.>}[d]\ar[r] & 0 \\
H^0(X_P,\mathcal F )  \ar[r]   &  H^0(\eta_P,\mathcal F )  \ar[r]  &  H^1_{x_P}(X_P, \mathcal F )  \ar[r]  & H^1(X_P, \mathcal F ) = 0
}
\]
where $M$ is multiplicity of $x_P$ in $D_P$. The dotted arrow is defined such that the diagram commutes and we let this dotted arrow be our 
isomorphism between the Nisnevich class group and the idele class group for $d=1$.

If $d\ge 2$ and $D_P=0$ we are done by the purity isomorphism~\eqref{purity}. 
Now assume $d\ge 2$ and $D_P\ne 0 $.
We have an isomorphism 
\begin{equation}
H^d_{x_P}(X_P,\mathcal F )  \cong H^{d-1}(X'_P,\mathcal F)
\end{equation}
where $X_P'=X_P-\{ x_P \}$.
Consider the coniveau spectral sequence
\[
E_1^{p,q} = \bigoplus_{x\in (X'_P)^{(p)}} H_x^{p+q} (X'_P, \mathcal F ) \Longrightarrow 
H^{p+q} (X'_P, \mathcal F ).
\]
By cohomological vanishing we have $E_1^{p,q}=0$ for $q\ge 1$. So the spectral sequence degenerates to an isomorphism between $\C_{\Nis} (X_P, \mathcal I_P )$ and the cokernel of
the map
\begin{equation}\label{loccoker}
\bigoplus_{x\in (X'_{P})_1} H^{d-2}_x (X'_P , \mathcal F ) \lr  
\bigoplus_{x\in (X'_{P})_0} H^{d-1}_x (X'_P , \mathcal F )  .
\end{equation}
If $d=2$ the map \eqref{loccoker} is isomorphic to the map
\begin{equation}\label{loccokeridele}
\bigoplus_{x\in (X'_{P})_1} \C (X_{(P,x)} , D_{(P,x)}  ) \lr  
\bigoplus_{x\in (X'_{P})_0} \C (X_{(P,x)} , D_{(P,x)}  )
\end{equation}
by our induction assumption. Here $D_{(P,x)}$ is the induced Weil divisor on $X_{(P,x)}$ and we denote the pullback of $\mathcal F$ 
to $X_{(P,x)}$ by the same symbol. But the cokernel of \eqref{loccokeridele} is just $\C(X_P, D_P)$ by definition. 

Now we assume $d\ge 3$.
For $x\in (X'_{P})_1$  we  have
\begin{equation}\label{eqcomp1}
 H^{d-2}_x (X'_P , \mathcal F ) \cong H^{d-2}_x(X_{(P,x)},\mathcal F ) \cong \C (  X_{(P,x)} ,D_{(P,x)} ) 
\end{equation}
and for $x\in (X'_{P})_0$ we have
\begin{equation}\label{eqcomp2} 
H^{d-1}_x (X'_P , \mathcal F ) \cong H^{d-1}_x(X_{(P,x)},\mathcal F ) \cong \C (  X_{(P,x)} ,D_{(P,x)} ) 
\end{equation}
by the induction assumption. 
In particular by its definition $\C (X_P ,D_P)$ is isomorphic to the cokernel of the map
\begin{equation}\label{eqcomp2a}
\bigoplus_{x\in (X'_{P})_1 \cap U_P} H^{d-2}_x (X_P , \mathcal F ) \lr  
\bigoplus_{x\in (X'_{P})_0} H^{d-1}_x (X_P , \mathcal F )  .
\end{equation}
In order to prove the theorem we have to show that the image of \eqref{eqcomp2a} is the same as the image of \eqref{loccoker}.

Fix $x$  in $(D_P)_2$. It follows from a local variant of Corollary~\ref{denseimage} that $\C (X_{(P,x)} ,D_{(P,x)} ) $ is generated by the images of
the groups $K^M_{d_P+3 } (k(P,x,y))=  \C (X_{(P,x,y)}, D_{(P,x,y)})$ where $y$ runs through the points of dimension $0$ in $U_{(P,x)}$.
Consider an element 
$$\alpha \in  \C (X_{(P,x)}, D_{(P,x)}) \cong  H^{d-2}_x (X'_P , \mathcal F ) $$
coming from $\C (X_{(P,x,y)}, D_{(P,x,y)})$  and denote by $y'$ the image of $y$ in $X_P$.  
Use a local variant of Corollary~\ref{corocontimm}
to find an element $\beta$ in  $K^M_{d_P+3} (k(P,y')) =\C (X_{(P,y')}, D_{(P,y')})$ such that the image of $\beta$ under the composite map
$$\C (X_{(P,y')}, D_{(P,y')}) \lr \C (X_{(P,x,y')}, D_{(P,x,y')}) \lr \C (X_{(P,x)}, D_{(P,x)})$$
is $\alpha$ and for all $x'\ne x$ in $(D_P)_2$ the image of $\beta$ in
$ \C (X_{(P,x')}, D_{(P,x')})$
vanishes.
Let $\gamma$ be the image of $\beta$ under the residue map
\[
\C (X_{(P,y')}, D_{(P,y')}) \lr \bigoplus_{x'\in (X'_P)_1 \cap U_P} \C (X_{(P,x')} , D_{(P,x')} ) \cong \bigoplus_{x\in (X'_{P})_1 \cap U_P} H^{d-2}_x (X_P , \mathcal F )  .
\]
By reciprocity the image of $\alpha$ in 
\[
 \bigoplus_{x\in (X'_{P})_0} H^{d-1}_x (X_P , \mathcal F ) 
\]
is the same as the image of $-\gamma$. But the latter lies in the image of \eqref{eqcomp2a}.

So we have shown that the natural surjective map
$ \C (X_P ,D_P)\to  \C_{\Nis} ( X_{P} ,\mathcal I_P )  $ stemming from the coniveau spectral sequence and the induction assumption is injective. We let this map be our natural isomorphism needed to conclude the induction step.
\end{proof}

\medskip

The global class group of Kato and Saito is defined as follows. Again $X$ is an integral scheme proper and flat over $\mathcal{O}_K$ of dimension
$d$.

\begin{defi}
Given an effective Cartier divisor with corresponding ideal sheaf $\mathcal I$ on $X$ whose associated Weil divisor is $D$ set
\[
\C_{\Nis}(X,\mathcal I)=H^d(X, \mathcal{K}^M_{d}(\mathcal{O}_{X}, \mathcal I) ) .
\]
\end{defi}

If $X-D$ is smooth over $\OO_K$ there is a triangle
\begin{equation}\label{globalcompdia}
\xymatrix{
 & \bigoplus_{x\in X_0} \Z  \ar[dl]_{\iota}  \ar[dr]^{\iota_\Nis} & \\
\C(X,D) \ar@{.>}[rr] & & \C_\Nis (X,\mathcal I)
}
\end{equation}
where $\iota$ is the obvious map, $\iota_\Nis$ is the map from  Theorem 2.5 of \cite{KS2} and the dotted arrow will be constructed 
in Theorem~\ref{Niscompteo2}.

Our global comparison theorem reads:

\begin{theo}\label{Niscompteo2}
If $U=X-D$ is smooth over $\mathcal{O}_K$ there exists a unique isomorphism
\[
\C_{\Nis}(X,\mathcal I) \cong \C (X, D) 
\]
making  diagram \eqref{globalcompdia} commute.
\end{theo}

\begin{proof}
Uniqueness is clear since the map $\iota$ in \eqref{globalcompdia} is surjective by Corollary~\ref{denseimage}.
Let $\mathcal F$ be the Nisnevich sheaf $ \mathcal{K}^M_{d}(\mathcal{O}_{X}, \mathcal I) $. As in the proof of Theorem~\ref{Niscomptheo1} the coniveau spectral
sequence
\[
E_1^{p,q} = \bigoplus_{x\in X^{(p)}} H_x^{p+q} (X, \mathcal F ) \Longrightarrow 
H^{p+q} (X, \mathcal F )
\]
degenerates to an isomorphism between $\C_{\Nis} (X, \mathcal I)$ and the cokernel of
the map
\begin{equation}\label{eqcoker1}
\bigoplus_{x\in X_1} H^{d-1}_x (X, \mathcal F ) \lr  
\bigoplus_{x\in X_0} H^{d}_x (X , \mathcal F )  .
\end{equation}
By Theorem~\ref{Niscomptheo1} we get an isomorphism between $\C (X,D)$ and the cokernel of the map
\begin{equation}\label{eqcoker2}
\bigoplus_{x\in U_1} H^{d-1}_x (X, \mathcal F ) \lr  
\bigoplus_{x\in X_0} H^{d}_x (X , \mathcal F )  .
\end{equation}
So we have to show that the maps (\ref{eqcoker1}) and (\ref{eqcoker2}) have the same image. For this it suffices to show that
the induced map 
\[
\bigoplus_{x\in D_1} \C (X_{(x)} ,D_{(x)} ) \lr \C (X,D)
\]
vanishes, which can be shown in a similar way as in the final part of the proof of Theorem~\ref{Niscomptheo1}.
\end{proof}

\section{Class field theory }\label{CFT}

\noindent In this section we restate the main results of Kato and Saito in terms of our new class group.
Fix a finitely generated field  $F$, which we assume for simplicity to contain a totally imaginary number field if $\chara(F)=0$. 
Let $X$ be an integral scheme proper over $\Z$ with $k(X)=F$ and endow $X$ with the dimension function $\di(x)=\dim( \overline{ \{ x \}} )$ for $x\in X$.

Kato's reciprocity law \cite[Proposition 7]{Kato} shows that the homomorphism $r : \I(F;X) \to \Gal (\bar{F} / F)^{ab}$
factors through a continuous homomorphism $\rho: \C(F;X) \to \Gal( \bar{F} / F)^{ab}$, which is called the reciprocity map.

\begin{theo}[Isomorphism]\label{isomorphism}
For a finite Galois extension $\phi:\Spec(K)\to \Spec(F)$ which extends to a finite morphism of integral schemes $Y\to X$ with
$k(Y)=K$
the map $\rho$ induces an isomorphism
\[
\C(F;X) / \phi_* \C(K;X) \tilde{\lr} \Gal(K/F)^{ab}  .
\]
\end{theo}

For a proof of the isomorphism theorem we refer to~\cite[Thm.~4.6]{KS2}.
The next theorem is the main theorem of Kato-Saito, see 
\cite[Theorem 9.1]{KS2}. 

\begin{theo}[Existence]\label{existence}
For any effective divisor $D$ on $X$ and for any integer $m>0$ the group
$\C (X,D)/ m$ is finite.
The reciprocity map $\rho$ induces an isomorphism of profinite groups
\[
\lim_{\longleftarrow \atop D,m} \C (X,D)/m \tilde{\lr}  \Gal (\bar{F} /F)^{ab}  .
\]
\end{theo}

If $\chara(F)=0$ we have an even stronger result due to Wiesend \cite{Wiesend} which comprises Theorem~\ref{existence} by 
\cite[Sec.~10]{Kerz}. For a complete proof
see~\cite{KSch} and~\cite{Kerz}.

\begin{theo} If $X$ is an integral scheme proper and flat over $\Spec(\Z) $
and $U\subset X$ is a regular open subscheme then the sequence
\[
0 \lr  \C_W(U)^0 \lr \C_W(U)  \lr \p (U) \lr 0
\]
is a strict exact sequence of topological groups.
\end{theo}

Here for a topological group $G$ we denote by $G^0$ the connected component of the identity element and strict exactness means in our case that the fundamental group
is the topological quotient of the other two groups.

\end{document}